\documentclass[reqno]{amsart}
\usepackage{amssymb,setspace}%
\usepackage[hmarginratio=1:1]{geometry}%
\usepackage{ifpdf}
\ifpdf
 \usepackage[hyperindex]{hyperref}
\else
 \expandafter\ifx\csname dvipdfm\endcsname\relax
 \usepackage[hypertex,hyperindex]{hyperref}
 \else
 \usepackage[dvipdfm,hyperindex]{hyperref}
 \fi
\fi
\allowdisplaybreaks[4]
\theoremstyle{plain}
\theoremstyle{definition}
\newtheorem{definition}{Definition}[section]

\newtheorem{theorem}{Theorem}[section]

\newtheorem{cor}{Corollary}[section]
\theoremstyle{remark}
\newtheorem{remark}{Remark}[section]

\numberwithin{equation}{section}

\begin{document}

\title[Radon-Nikodym theorem with respect to $p$-adic $(\rho,q)$-measure on $\Bbb Z_p$]
	{Radon-Nikodym theorem with respect to $(\rho,q)$-measure on $\Bbb Z_p$}

\author[D. Lim]{Dongkyu Lim}
\address[Lim]{Department of Mathematics Education, Andong National University, Andong 36729, Republic of Korea}
\email{\href{mailto: D. Lim <dklim@anu.ac.kr>}{dklim@anu.ac.kr}}
\urladdr{\url{http://orcid.org/0000-0002-0928-8480}}


\begin{abstract}
Araci {\it et al.} introduced a $p$-adic $(\rho,q)$-analogue of the Haar distribution. By means of the distribution, they constructed the $p$-adic $(\rho,q)$-Volkenborn integral. In this paper, by virtue of the Mahler expansion of continuous functions, the author gives the Radon-Nikodym theorem with respect to the $p$-adic $(\rho,q)$-distribution on $\Bbb Z_p$.
\end{abstract}




\maketitle


\section{\bf Introduction}

The main object in this paper is the $p$-adic distribution
\begin{equation*}
	\mu_{\rho,q}(a+p^N\mathbb{Z}_p)=\frac{\rho^{p^N}}{[p^N]_{\rho,q}}\biggl(\frac{q}{\rho}\biggl)^{a}
\end{equation*}
which was introduced by Araci {\it et al.} \cite{Ara}. The $p$-adic $(\rho,q)$-distribution is an analogue of Haar distribution based on $(\rho,q)$-numbers. They used it to construct the $p$-adic $(\rho,q)$-Volkenborn integral of a function $f \in UD(\Bbb Z_p)$ as follows:
\begin{equation*}
	I_{\rho,q} (f)=\int_{\Bbb
		Z_p} f(x) d \mu_{\rho,q} (x)= \lim_{N \rightarrow \infty}
	\frac{\rho^{p^N}}{[p^N]_{\rho,q}} \sum_{x=0}^{p^N -1} f(x) \biggl(\frac{q}{\rho}\biggl)^x.
\end{equation*}
The application of the $(\rho,q)$-Volkenborn is an effective way to deduce many important results in $(\rho,q)$-numbers and polynomials. 

The original motivation of this paper is the following theorem of Radon-Nikodym. 
\begin{theorem}\label{LRT}[\cite{Bauer}, p.~101, Theorem 17.10]
	Let $\mu$ and $\nu$ be measures on a $\sigma$-algebra in a set $X$. If $\mu$ is $\sigma$-finite, the following two assertions are equivalent:
	\begin{enumerate}
		\item $\nu$ has a density with respect to $\mu$.
		\item $\nu$ is $\mu$-continuous.
	\end{enumerate} 
\end{theorem}
The above theorem was proved in 1930 by O. Nikodym \cite{Niko}. H. Lebesgue proved the theorem in 1910 for the case where $\mu$ is the Lebesgue-Borel measure. J. Radon pushed things further in a fundamental work which appeared in 1913. So Theorem \ref{LRT} is called the theorem of Lebesgue-Radon-Nikodym.
The Radon-Nikodym theorem is a very important tool in the study of Measure Theory since it permits us to decompose a given measure in terms of two others that keep some properties from each other. It is a deep result which demands a previous study of several concepts to the definition of measures like the $\sigma$-algebras, exterior measures, etc.
Various proofs of the Radon-Nikodym theorem can be found in many books on measure theory, analysis, or probability theory. Usually, they use the Hahn decomposition theorem for signed measures, the Riesz representation theorem for functionals on Hilbert spaces, or the martingale theory, (see \cite{Bauer,L1,L2,L3,Rud,Roy,L4}).

The $p$-adic numbers were first introduced by K. Hensel in 1897, though with hindsight some of Kummer's earlier work can be interpreted as implicitly using $p$-adic numbers. The $p$-adic numbers were motivated primarily by an attempt to bring the ideas and techniques of power series methods into number theory. Their influence now extends far beyond this. For example, the field of $p$-adic analysis essentially provides an alternative form of calculus. 
The one important tool of these investigations is $p$-adic integral (Volkenborn integral) which is firstly described by Arnt Volkenborn \cite{Vol-1,Vol-2} in about 1972. The $p$-adic integral was used in mathematical physics, for instance, the functional equations of the Zeta functions, Stirling numbers, and Mahler theory of integration with respect to the ring $\mathbb{Z}_p$ together with Iwasawa’s $p$-adic $L$ functions {\it cf.} \cite{Kob}.

The $q$-calculus in the field of special functions was investigated in the last two decades
(see, e.g., \cite{Car,Choi,Do,DKim,Kim1,Ozd,Sim}). Many generalizations of special functions with a $q$ parameter were presented using $p$-adic $q$-integral on $\mathbb{Z}_p$ ($q$-Volkenborn
integral) introduced by Kim \cite{Kim1}. The $p$-adic $q$-integral on $\mathbb{Z}_p$ is a powerful tool to construct various generating functions of special polynomials and to obtain various explicit formulae. For more information refer to \cite{Do,DKim,Ozd,Sim} and the closely related references therein. Further, Kim \cite{Kim2} derived the Lebesgue-Radon-Nikodym theorem with respect to $p$-adic $q$-invariant distribution.

The $(\rho,q)$-integer was introduced to generalize or unify several forms of $q$-oscillator
algebras are well-known in the physics literature related to the representation theory of single parameter quantum algebras. In \cite{Ara}, Araci {\it et al.} introduced $(\rho,q)$-analogue of the Haar distribution and by means of the distribution, they constructed $(\rho,q)$-Volkenborn integral yielding to Carlitz-type $(\rho,q)$-Bernoulli numbers and polynomials. By using the integral, Duran {\it et al.} defined $(\rho,q)$-Euler numbers and polynomials in \cite{Duran}. The purpose of this paper is to derive the analogue of the Radon-Nikodym theorem with respect to $(\rho,q)$-measure on $\Bbb Z_p$.
\medskip

\section{\bf Preliminary}

In this section, we give the necessary definitions and notations.

Let $p$ be a fixed odd prime number. Throughout this paper, the
symbol $\Bbb Z$, $\Bbb Z_p$, $\Bbb Q_p$, and $\Bbb C_p$ denote the
ring of rational integers, the ring of $p$-adic integers, the field
of $p$-adic rational numbers, and the completion of algebraic
closure of $\Bbb Q_p$, respectively. Let $\nu_p $ be the normalized
exponential valuation of $\Bbb C_p$ with $|p|=
p^{-\nu_{p}(p)}=p^{-1}$. 

Suppose that $d$ is an positive number with $(d,p)=1$. Set
\begin{equation*}
	X:=X_d=\varprojlim_N\mathbb{Z}/dp^N\mathbb{Z}\ \ \text{and}\ \ X_1=\mathbb{Z}_p
\end{equation*}	
and
\begin{equation*}
	a+dp^N\mathbb{Z}_p=\{x\in X \vert x\equiv a\!\!\!\!\pmod{dp^N}  \}
\end{equation*}	
where $a\in\mathbb{Z}$ lies in $0\leq a<dp^N$.

\begin{definition}
	A function $f$ is a uniformly differentiable, denoted by $f\in UD(\mathbb{Z}_p)$, if for a given function $f:\mathbb{Z}_p\rightarrow\mathbb{C}_p$, there exists a continuous function $F_f(x,y)\rightarrow\mathbb{C}_p$ 
\begin{equation*}
	F_f(x,y)=\frac{f(x)-f(y)}{x-y},
\end{equation*}
for all $x,y\in\mathbb{Z}_p$, $x\neq y$.
\end{definition}

In this paper, we assume that  $\rho$ and $ q \in \Bbb
C_p$ with $|\rho-1| < p^{- \frac{1}{p-1}}$ and $|q-1| < p^{- \frac{1}{p-1}}$ as an indeterminate. The $q$-analogue of a fixed number $x\in \Bbb Z_p$ is that
\begin{equation*}
	[x]_q = [x:q]=\frac{1- q^x }{1-q}.
\end{equation*}
Observe that $\lim\limits_{q\rightarrow1}[x]_q=x$ for arbitrary $x$ with $|x|\leq1$.

The $(\rho,q)$-integer was introduced to generalized or unify several forms of $q$-oscillator algebras well-known in the pyhsics literature related to the representation theory of single parameter quantum algebras, in which the $(\rho,q)$-integer is given by 
\begin{equation*}
	[n]_{\rho,q} = \frac{\rho^n- q^n }{\rho-q}
\end{equation*}
which is a native generalization of the $q$-number, such that
\begin{equation*}
	\lim\limits_{\rho\rightarrow 1}[n]_{\rho,q} := [n]_q.
\end{equation*}
Note that the notation $[n]_{\rho,q}$ is symmetric. We also use the following notation for $(\rho,q)$-series:
\begin{equation*}
 \begin{aligned}
	{n \brack k}_{\rho,q}=\frac{[n]_{\rho,q}!}{[n-k]_{\rho,q}![k]_{\rho,q}!},\quad  [n]_{\rho,q}!=[n]_{\rho,q}[n-1]_{\rho,q}\ldots[2]_{\rho,q}[1]_{\rho,q}.
 \end{aligned}
\end{equation*}





\begin{definition}\cite{Kob}
	A $p$-adic distiribution $\mu$ on $X$ is $\mathbb{Q}_p$-linear vector space homomorphism from the $\mathbb{Q}_p$-vector space of locally constant functions $X$ to $\mathbb{Q}_p$. If $f:X\rightarrow\mathbb{Q}_p$ is locally constant function, instead of writing $\mu(f)$ for the value of $\mu$ at $f$, we usually write $\int f \mu$.
\end{definition}

First, we recall from \cite{Ara} the $p$-adic $(\rho,q)$-distribution on $\mathbb{Z}_p$.
\begin{definition}
	For $N\geq1$, $p$-adic $(\rho,q)$-distribution $\mu_{\rho,q}$ on $X$ is defined by
	\begin{equation*}
	\mu_{\rho,q}(a+p^N\mathbb{Z}_p)=\frac{\rho^{p^N}}{[p^N]_{\rho,q}}\biggl(\frac{q}{\rho}\biggl)^{a}.
	\end{equation*}
\end{definition}

\begin{definition}\label{def3}\cite{Ara}
	The $p$-adic $(\rho,q)$-Volkenborn integral of a function $f \in UD(\Bbb Z_p)$ is defined by
	\begin{equation*}
	I_{\rho,q} (f)=\int_{\Bbb
		Z_p} f(x) d \mu_{\rho,q} (x)= \lim_{N \rightarrow \infty}
	\frac{\rho^{p^N}}{[p^N]_{\rho,q}} \sum_{x=0}^{p^N -1} f(x) \biggl(\frac{q}{\rho}\biggl)^x.
	\end{equation*}
\end{definition}

\begin{remark}
In the limit as $\rho\rightarrow1$ in Definition \ref{def3} becomes the $q$-Volkenborn integration \cite{Kim1}.
\end{remark}

Recall from \cite{Ara} that the Caritz's-type $(\rho, q)$-Bernoulli numbers $\beta_{n:a}(\rho,q)$ may be represented by
\begin{equation*}
	\beta_{n:a}(\rho,q)=\int_{\mathbb{Z}_p}\rho^{ax}[x]^{n}_{\rho,q}d\mu_{\rho,q}(x).
\end{equation*}	
Specially, $$\beta_{0,a}(\rho,q)=\frac{k\log{\rho}}{\log{\rho q}}.$$


By the meaning of the extension of $({\rho,q})$-Volkenborn integral, we consider the below weakly (strongly) $p$-adic $({\rho,q})$-invariant distribution $\mu_{\rho,q}$ on $\mathbb{Z}_p$.
\begin{definition}
	
	If $\mu_{\rho,q}$ is satisfied the
	following equation:
	$$ \big|  [p^n]_{\rho,q}\mu_{\rho,q}( a+ p^n \Bbb Z_p)-[p^{n+1}]_{\rho,q}\mu_{\rho,q}( a+ p^{n+1} \Bbb Z_p)  \big|   \leq  \delta_{n,\rho,q},$$
	where $ \delta_{n,\rho,q} \rightarrow 0$ and $n  \rightarrow \infty$
	and $ \delta_{n,\rho,q}$ is independent of $a$, then $ \mu_{\rho,q}$ is
	called {the weakly $p$-adic $({\rho,q})$-distribution on $\Bbb Z_p$}.
	If $ \delta_{n,{\rho,q}}$ is replaced by $C
	p^{-\nu_p(\rho^{p^n}-q^{p^n})}$ ($C$ is some constant),  then $\mu_{\rho,q}$
	is called {the strongly $p$-adic $({\rho,q})$-distribution on $\Bbb Z_p$}.	
\end{definition}

\begin{definition}
	A distribution $\mu_{\rho,q}$ on $\mathbb{Z}_p$ is called $1$-admissible if
		\begin{equation*}
		|\mu_{\rho,q}( a+ p^{N} \Bbb Z_p)| \leq \left\vert\frac{\rho^{p^N}}{[p^N]_{\rho,q}}\right\vert c_n,
	\end{equation*} 
where $c_n\rightarrow 0$.
\end{definition}
\begin{remark}
	If $\mu_{\rho,q}$ is $1$-admissible on $\Bbb Z_p$, then we see that $\mu_{\rho,q}$ is a weakly $p$-adic $({\rho,q})$-invariant distribution on $\Bbb Z_p$.
\end{remark}

Let $C(\Bbb Z_p, \Bbb C_p)$ be the space of continuous function on $\Bbb Z_p$ with values in $\Bbb C_p$, provided with norm $||f||_\infty=\underset{x \in \Bbb Z_p}{\sup} |f(x)|$. The difference quotient $\bigtriangleup_1 f$ of $f$ is the function of two variables given by $$\bigtriangleup_1 f(m, x) = \frac{f(x+m)-f(x)}{m}, $$ for all $x \in \Bbb Z_p$ and $m \in \Bbb Z_p$ with $m \neq 0$.
\begin{definition}
	A function $f:\Bbb Z_p \longrightarrow \Bbb C_p$ is said to be a Lipschitz function if there exists a constant $M>0$  such that for
	all $m \in \Bbb Z_p \setminus \{0 \}$ and $x \in \Bbb Z_p$, $$|\bigtriangleup_1 f(m, x)| \le M.$$
	Here $M$ is called the Lipschitz constant of $f$. The $\Bbb C_p$-linear space consisting of all Lipschitz function is denoted by Lip$(\Bbb Z_p, \Bbb C_p)$.
\end{definition}

\begin{remark}
	$\Bbb
	C_p$-linear space is a Banach space with respect
	to the norm $||f||_1 = ||f||_\infty \vee ||\bigtriangleup_1
	f||_\infty$.
\end{remark}
\medskip


\section{\bf Radon-Nikodym's type theorem with respect to $(\rho,q)$-measure on $\Bbb Z_p$}

Let $\mu_{\rho,q}$ be a weakly $p$-adic $(\rho,q)$-invariant distribution, then $\lim_{N\rightarrow\infty}[p^{n}]_{\rho,q}\mu_{\rho,q}(x+p^N\mathbb{Z}_p)$ exists for all $x\in\mathbb{Z}_p$, and converges uniformly with respect to $x$. We define the $(\rho,q)$-analogue of Radon-Nikodym derivative of $\mu_{\rho,q}$ with respect to $p$-adic $(\rho,q)$-invariant distribution as follows:
	\begin{equation*}
f_{\mu_{ \rho,q}}(x)=\lim\limits_{N\rightarrow\infty}[p^{N}]_{\rho,q}\mu_{\rho,q}(x+p^N\mathbb{Z}_p).
\end{equation*}

\begin{theorem} 
Let $\mu_{\rho,q}$ is strongly $(\rho,q)$-invariant distribution on $\mathbb{Z}_p$. Then $f_{\mu_{ \rho,q}}\in Lip(\mathbb{Z}_p,\mathbb{C}_p)$.
\end{theorem}

\begin{proof}
Let $x,y\in\mathbb{Z}_p$ with $\vert x-y\vert_p\leq p^{-m}$. Then $x+p^m\mathbb{Z}_p=y+p^m\mathbb{Z}_p$. Hence, we note that
\begin{equation*}
\begin{aligned}
\vert f_{\mu_{\rho,q}}(x)-f_{\mu_{\rho,q}}(y) \vert = &\vert  f_{\mu_{\rho,q}}(x)-[p^{n}]_{\rho,q}\mu_{\rho,q}(x+p^n\mathbb{Z}_p)-f_{\mu_{\rho,q}}(y)+[p^{n}]_{\rho,q}\mu_{\rho,q}(y+p^n\mathbb{Z}_p)\\
&+[p^{n}]_{\rho,q}\mu_{\rho,q}(x+p^n\mathbb{Z}_p)-[p^{m}]_{\rho,q}\mu_{\rho,q}(x+p^m\mathbb{Z}_p)-[p^{n}]_{\rho,q}\mu_{\rho,q}(y+p^n\mathbb{Z}_p)\\
&+[p^{m}]_{\rho,q}\mu_{\rho,q}(y+p^m\mathbb{Z}_p)\vert.
\end{aligned}
\end{equation*}	

Since $\mu_{\rho,q}$ is strongly $(\rho,q)$-invariant distribution on $\mathbb{Z}_p$, we have
\begin{equation*}
	\begin{aligned}
	\vert  [p^{n}]_{\rho,q}\mu_{\rho,q}(x+p^n\mathbb{Z}_p)-[p^{m}]_{\rho,q}\mu_{\rho,q}(x+p^m\mathbb{Z}_p)\vert\leq Cp^{-m},
	\end{aligned}
\end{equation*}	
for some constant $C$ and $m<n$.
For $n\gg0$, we obtain
\begin{equation*}
	\begin{aligned}
		\vert f_{\mu_{\rho,q}}(x)-f_{\mu_{\rho,q}}(y) \vert \leq C_1p^{-m}=C_1 \vert x-y\vert,
	\end{aligned}
\end{equation*}	
where $C_1$ is positive real constant. Thus this completes the proof.
\end{proof}

For any positive integer $a$ and $n$ with $a < p^n$,  and $f \in UD ( \Bbb Z_p)$, let us define
\begin{equation}\label{sec:eq1}
\begin{aligned}
\Tilde{\mu}_{f,\rho,q}( a+ p^n \Bbb Z_p) =\int_{a+p^n\mathbb{Z}_p} f(x) d\mu_{\rho,q}(x),
\end{aligned}
\end{equation}	
where the integral is the $(\rho,q)$-Volkenborn integral on $\Bbb Z_p$.

\begin{theorem} \label{thm1}
	For $ f, g \in UD( \Bbb Z_p)$, we have
	\begin{enumerate}
		\item $\Tilde{\mu}_{ \alpha f + \beta g, \rho,q}= \alpha \Tilde{\mu}_{  f, \rho,q} + \beta  \Tilde{\mu}_{g,\rho,q},$ 
		\item $\vert\Tilde{\mu}_{\rho,q}( a+ p^{n} \Bbb Z_p)\vert \leq \|f\|_1\left\vert\left(\frac{q}{\rho}\right)^a \right\vert \left\vert\frac{1}{[p^n]_{\rho,q}}\right\vert.$
	\end{enumerate}

\end{theorem}

\begin{proof}

	From \eqref{sec:eq1}, these follow from the observation that
	
	\begin{equation*}
	\begin{aligned}
	\Tilde{\mu}_{f,\rho, q}( a+ p^n \Bbb Z_p) &=\lim_{m \to \infty}\dfrac{{\rho}^{p^{m}}}{[p^{m+n}]_{\rho,q}}\sum_{x=0}^{
		p^m-1}f(a+p^nx) \biggl(\frac{q}{\rho}\biggl)^{a+p^nx}\\
	&=\lim_{m \to \infty}\dfrac{{\rho}^{p^m}}{[p^{m}]_{\rho,q}}\sum_{x=0}^{
		p^{m-n}-1}f(a+p^nx) \biggl(\frac{q}{\rho}\biggl)^a\biggl(\frac{q}{\rho}\biggl)^{p^nx}\\
	&=\dfrac{1}{[p^n]_{\rho,q}}\biggl(\frac{q}{\rho}\biggl)^a\lim_{m \to \infty}\dfrac{{\rho}^{p^m-n}}{[p^{m-n}]_{\rho^{p^m},q^{p^m}}}\sum_{x=0}^{
		p^{m-n}-1}f(a+p^nx)\left(\biggl(\frac{q}{\rho}\biggl)^{p^n}\right)^x\\
	&= \dfrac{1}{[p^n]_{\rho,q}}\biggl(\frac{q}{\rho}\biggl)^a \int_{\mathbb{Z}_p}
	f(a+p^nx) d\mu_{\rho^{p^n},q^{p^n}}(x).
	\end{aligned}
	\end{equation*}

\end{proof}

\begin{cor}
For $ f \in UD( \Bbb Z_p)$, we have
\begin{equation*}
	\begin{aligned}
		\big| \Tilde{\mu}_{f, {\rho,q}}( a+ p^n \Bbb Z_p) \big|   \leq M \| f
		\|_{\infty},
	\end{aligned}
\end{equation*}
where $ \| f \|_{\infty}= \sup_{x \in \Bbb	Z_p} |f(x)|$ and $M$ is some positive constant.	
\end{cor}

\begin{theorem}\label{thm2} 
	Let $P(x)\in \Bbb C_p[[x]_{\rho,q}]$ be an arbitrary $({\rho,q})$-polynomial with
	$ \sum a_i [x]_{\rho,q}^i$. Then   $\mu_{P,{\rho,q}}$ is a strongly $p$-adic $({\rho,q})$-distribution on $\Bbb Z_p$ and for all $x \in \Bbb Z_p$,
	\begin{equation*}
	\begin{aligned}
	f_{ \Tilde{\mu}_{P,\rho,q}}(x) = \biggl(\frac{q}{\rho}\biggl)^{a}  P(x)\beta_{0:k}(\rho,q),
	\end{aligned}
	\end{equation*}
where $k$ is the degree of $P(x)$.
	Furthermore, for $g \in UD( \Bbb Z_p)$, we have
	\begin{equation*}
	\begin{aligned}
	\int_{\mathbb{Z}_p} g(x) d\Tilde{\mu}_{P,\rho,q}(x)=\beta_{0:k}(\rho,q)\int_{\mathbb{Z}_p}   g(x)
	P(x) d\mu_{\rho,q}(x),
	\end{aligned}
	\end{equation*}
	where the second integral is  $(\rho,q)$-Volkenborn integral.
\end{theorem}

\begin{proof}
	Since $P(x)\in \Bbb C_p[[x]_{\rho,q}]$ be an arbitrary
	$({\rho,q})$-polynomial with $ \sum a_i [x]_{\rho,q}^i$, then
	$\mu_{P,{\rho,q}}$ is strongly $(\rho,q)$-measure on $\Bbb
	Z_p$. Without a loss of generality, it is enough to prove the
	statement for $P(x)=[x]_{\rho,q}^k$.

	Let $a$ be an integer with $0 \leq a < p^n$. Then we have
	\begin{equation*}\label{sec:eq5}
	\begin{aligned}
	\Tilde{\mu}_{P,\rho,q}(a +p^n \Bbb Z_p) =\dfrac{1}{[p^n]_{\rho,q}} \biggl(\frac{q}{\rho}\biggl)^{a} \lim_{m \to \infty}\dfrac{\rho^{p^{m-n}}}{[p^{m-n}]_{\rho^{p^m},q^{p^m}}}
	\sum_{i=0}^{ p^{m-n}-1}  [a+ i p^n]_{\rho,q}^k \biggl(\frac{q}{\rho}\biggl)^{p^ni}
	\end{aligned}
	\end{equation*}
	and
\begin{equation*}
	\begin{aligned}
	[] [a+ip^n]^k_{\rho,q}&=\left(q^a[ip^n]_{\rho,q}+\rho^{ip^n}[a+{\rho,q}]\right)^k\\
	&=\left(q^a[p^n]_{\rho,q}[i]_{{\rho^p}^n,{q^p}^n}+\rho^{ip^n}[a]_{\rho,q}\right)^k\\
	&=\sum_{l=0}^{k}\binom{k}{l}q^{al}[p^n]^l_{\rho,q}[i]^l_{{\rho^p}^n,{q^p}^n}\rho^{ip^n(k-l)}[a]^{k-l}_{\rho,q}.
	\end{aligned}
\end{equation*}	
This implies that	
\begin{equation*}
	\begin{aligned}
		\Tilde{\mu}_{P,\rho,q}(a+p^n\mathbb{Z}_p)&=\frac{1}{[p^n]_{\rho,q}} \biggl(\frac{q}{\rho}\biggl)^{a}\{[a]^k_{\rho,q}\beta_{0:k}({\rho^p}^n,{q^p}^n)+q^a[p^n]_{\rho,q}k\beta_{1:k-1}({\rho^p}^n,{q^p}^n)+\cdots\\
		&\quad +[p^n]^{k}_{\rho,q}q^{ak}\beta_{k:0}({\rho^p}^n,{q^p}^n)\}\\
		&= \sum_{i=0}^{\infty}\frac{P^{(i)}(a)}{i!}\beta_{i:k-i}(\rho^{p^n},q^{p^n})[p^n]^i_{\rho,q}q^{ai},
	\end{aligned}
\end{equation*}	
where $P^{(i)}(a)=\left(\frac{d}{d[x]_{\rho,q}}\right)^iP(x)\big\vert_{x=a}$.	
Thus we acquire
\begin{equation*}\label{sec:eq6}
	\begin{aligned}
[]	[p^n]_{\rho,q}\Tilde{\mu}_{P,\rho,q}(a+p^n\mathbb{Z}_p)&=\biggl(\frac{q}{\rho}\biggl)^{a}\{[a]^k_{\rho,q}\beta_{0:k}({\rho^p}^n,{q^p}^n)+[p^n]_{\rho,q}q^ak[a]^{k-1}_{\rho,q}\beta_{1:k-1}({\rho^p}^n,{q^p}^n)+\cdots\\
	&\ \ +[p^n]^{k-1}_{\rho,q}q^{ak}\beta_{k:0}({\rho^p}^n,{q^p}^n)\}.
	\end{aligned}
\end{equation*}
This can be written in the form
	\begin{equation*}\label{sec:eq7}
	\begin{aligned}
	\Tilde{\mu}_{P,\rho,q}&(a +p^n \Bbb Z_p) \\ & \equiv \dfrac{1}{[p^n]_{\rho,q}}\biggl(\frac{q}{\rho}\biggl)^{a}[a]_{\rho,q}^k\beta_{0:k}(\rho^{p^n},q^{p^n})+\biggl(\frac{q}{\rho}\biggl)^{a}q^a[a]^{k-1}_{\rho,q}\beta_{1:k-1}(\rho^{p^n},q^{p^n}) \pmod {[p^n]_{\rho,q}}\\
	& \equiv \dfrac{1}{[p^n]_{\rho,q}}\biggl(\frac{q}{\rho}\biggl)^{a} P(a)\beta_{0:k}(\rho,q)\beta_{0:k}(\rho^{p^n},q^{p^n}) +q^ak[a]^{k-1}_{\rho,q}\beta_{1:k-1}(\rho^{p^n},q^{p^n})\pmod {[p^n]_{\rho,q}}.
	\end{aligned}
	\end{equation*}
		
	Let $x$ be an arbitrary in $\Bbb Z_p$ with $x \equiv x_n
	\pmod {p^n}$ and $x \equiv x_{n+1} \pmod {p^{n+1}}$,  where $x_n$
	and $x_{n+1}$ are positive integers such that $ 0 \leq x_n < p^n$
	and $ 0 \leq x_{n+1} < p^{n+1}$. Then
	\begin{equation*}
	\begin{aligned}
		\big|&[p^{n}]_{\rho,q}  \Tilde{\mu}_{P,\rho, q}( x+ p^n \Bbb Z_p)-[p^{n+1}]_{\rho,q}\mu_{P,\rho, q}( x+ p^{n+1} \Bbb Z_p)  \big| \\
		&=\biggl|\sum_{i=0}^{\infty}\frac{P^{(i)}(x_n)}{i!}\beta_{i:k-i}(\rho^{p^{n}},q^{p^n})[p^n]^i_{\rho,q}q^{ai}-\sum_{i=0}^{\infty}\frac{P^{(i)}(x_{n+1})}{i!}\beta_{i:k-i}(\rho^{p^{n+1}},q^{p^{n+1}})[p^n]^i_{\rho^p,q^p}[p^n]^i_{\rho,q}q^{ai}   \biggl|.
	\end{aligned}
\end{equation*}
Therefore, we obtain
	\begin{equation*}
	\begin{aligned}
	\big|[p^{n}]_{\rho,q} \Tilde{\mu}_{P,\rho, q}( a+ p^n \Bbb Z_p)-[p^{n+1}]_{\rho,q}\Tilde{\mu}_{P,\rho, q}( a+ p^{n+1} \Bbb Z_p)  \big|   \leq  C
	p^{-\nu_p(\rho^{p^n}-q^{p^n})},
	\end{aligned}
	\end{equation*}
	where $C$ is a positive some constant and $ n \gg 0$.

	Let 
	\begin{equation*}\label{sec:eq8}
	\begin{aligned}
	f_{ \Tilde{\mu}_{P, \rho, q}}(a)= \lim_{n \rightarrow \infty}
	[p^n]_{\rho,q}\Tilde{\mu}_{P, \rho, q}( a+ p^n \Bbb Z_p).
	\end{aligned}
	\end{equation*}
	Then, we have
	\begin{equation*}
	\begin{aligned}
	f_{ \Tilde{\mu}_{P,\rho,q}}(a)  =\biggl(\frac{q}{\rho}\biggl)^{a} P(a)\beta_{0:k}(\rho,q).
	\end{aligned}
	\end{equation*}
	Furthermore, since $f_{ \Tilde{\mu}_{P,\rho,q}}(x)$ is continuous on $\Bbb
	Z_p$, it follows for all $x \in \Bbb Z_p$
	\begin{equation}\label{sec:eq10}
	\begin{aligned}
	f_{ \Tilde{\mu}_{P,\rho,q}}(x) = \biggl(\frac{q}{\rho}\biggl)^{x} P(x)\beta_{0:k}(\rho,q).
	\end{aligned}
	\end{equation}
	Suppose that $g \in UD( \Bbb Z_p)$. From \eqref{sec:eq10}, we obtain
	\begin{equation*}
	\begin{aligned}
	\int_{\mathbb{Z}_p} g(x) d\Tilde{\mu}_{P,\rho,q}(x)& = \lim_{n \to \infty}
	\sum_{i=0}^{ p^{n}-1} g(i)  \Tilde{\mu}_{P,\rho,q}(i+  p^n\Bbb Z_p)\frac{{\rho^p}^n}{[p^n]_{\rho,q}} \\
	& =\beta_{0:k}(\rho,q)\lim_{n\rightarrow \infty} \sum_{i=0}^{ p^{n}-1} g(i)  \biggl(\frac{q}{\rho}\biggl)^{i}P(i)\frac{(\rho^p)^n}{[p^n]_{\rho,q}}\\
	&=  \int_{\mathbb{Z}_p} g(x) P(x) d\mu_{\rho,q}(x).
	\end{aligned}
	\end{equation*}
	The proof of Theorem \ref{thm2} is complete.
\end{proof}
\bigskip

\begin{theorem}\label{thm3} 
	Let  $\mu_{\rho,q}$ be  a strongly  $({\rho,q})$-measure on
	$\Bbb Z_p$, and assume that the  $({\rho,q})$- Radon-Nikodym
	derivative $f_{ \mu_{\rho,q}}$ on  $\Bbb Z_p$ is a continuous function
	on  $\Bbb Z_p$. Suppose that $\mu_{1, \rho,q}$ is the strongly
	$({\rho,q})$-measure associated to $f_{ \mu_{\rho,q}}$, then
	there exists  a $({\rho,q})$-measure $\mu_{2, \rho,q}$ on  $\Bbb Z_p$ such that
	\begin{equation*}
	\begin{aligned}
	\mu_{\rho,q} = \mu_{1, \rho,q} + \mu_{2, \rho,q}.
	\end{aligned}
	\end{equation*}
\end{theorem}

\begin{proof}
	
	Let $f(x)= \sum_{n=0}^\infty a_{n,\rho,q} {x \brack n}_{\rho,q}$ be
	the $(\rho,q)$-Mahler expansion of continuous function $f$ on $\Bbb Z_p$, then we note that $\lim_{n \to \infty} | a_{n, \rho,q}|=0$. 
	
	Suppose that 
	\begin{equation*}
	\begin{aligned}
	f_m(x)= \sum_{i=0}^m a_{i,\rho,q}  {x \brack i}_{\rho,q}  \in
	\Bbb C_p[[x]_{\rho,q}].
	\end{aligned}
	\end{equation*}
	Then it is easily checked the following inequality
	\begin{equation*}
	\begin{aligned}
	\|  f-f_m \|_{\infty} \leq \sup_{m \leq n} | a_{n,\rho,q} |.
	\end{aligned}
	\end{equation*}
	Writing $f= f_m +f-f_m$, it follows that
	\begin{equation*}
	\begin{aligned}
	& \big| [p^{n}]_{\rho,q} \mu_{f,\rho,q}( a+ p^n \Bbb Z_p)-[p^{n+1}]_{\rho,q}\mu_{f,\rho,q}( a+ p^{n+1} \Bbb Z_p)  \big|  \\
	&  \leq \max \{ \big| [p^{n}]_{\rho,q} \mu_{f_m,\rho,q}( a+ p^n \Bbb
	Z_p)-[p^{n+1}]_{\rho,q}\mu_{f_m,\rho,q}( a+ p^{n+1} \Bbb Z_p)  \big|, \\
	& \quad \quad  \qquad   \big|[p^{n}]_{\rho,q} \mu_{f-f_m,\rho,q}( a+ p^n \Bbb
	Z_p)-[p^{n+1}]_{\rho,q}\mu_{f-f_m,\rho,q}( a+ p^{n+1} \Bbb Z_p) \big| \}.
	\end{aligned}
	\end{equation*}
	We conclude by Theorem \ref{thm2} that
	\begin{equation}\label{sec:eq12}
	\begin{aligned}
	\big| [p^{n}]_{\rho,q}\mu_{f-f_m,\rho,q}( a+ p^n \Bbb Z_p) \big|   \leq  \| f
	-f_m \|_{\infty} \leq C_1 p^{-\nu_p(\rho^{p^n}-q^{p^n})} ,
	\end{aligned}
	\end{equation}
	where $C_1$ is some positive constant.	
	For  $ m \gg 0$, we have $  \| f \|_{\infty} = \|f_m
	\|_{\infty} .$	Accordingly,
	\begin{equation}\label{sec:eq13}
	\begin{aligned}
	\big| [p^{n}]_{\rho,q} \mu_{f_m,\rho,q}( a+ p^n \Bbb Z_p)-[p^{n+1}]_{\rho,q}\mu_{f_m,\rho,q}( a+ p^{n+1} \Bbb Z_p)  \big|  \leq C_2 p^{-\nu_p(\rho^{p^n}-q^{p^n})} ,
	\end{aligned}
	\end{equation}
	where $C_2$ is also  some positive constant. This implies, by virtue of \eqref{sec:eq12} and \eqref{sec:eq13},
	\begin{equation*}
	\begin{aligned}
	& \big|  f(a)- [p^{n}]_{\rho,q}\mu_{f,\rho,q}( a+ p^n \Bbb Z_p)  \big| \\
	&  \leq \max \{ \big| f(a)- f_m(a) \big|, \big| f_m(a)-[p^{n}]_{\rho,q} \mu_{f_m,\rho,q}( a+ p^n \Bbb Z_p)\big|, \big|[p^{n}]_{\rho,q}\mu_{f-f_m,\rho,q}( a+ p^{n} \Bbb Z_p)  \big| \} \\
	&  \leq \max \{ \big| f(a)- f_m(a) \big|, \big| f_m(a)- [p^{n}]_{\rho,q}\mu_{f_m,
		\rho,q}( a+ p^n \Bbb Z_p)\big|,   \| f
	-f_m \|_{\infty} \}.
	\end{aligned}
	\end{equation*}
		If we fix $ \epsilon >0$ and fix $m$ such that $ \| f
	-f_m \| \leq \epsilon$, then for $ n \gg 0$, we have
	\begin{equation*}
	\begin{aligned}
	\big|  f(a)- [p^{n}]_{\rho,q}\mu_{f,\rho,q}( a+ p^n \Bbb Z_p)  \big|  \leq  \epsilon.
	\end{aligned}
	\end{equation*}
	Hence, we acquire
	\begin{equation*}
	\begin{aligned}
	f_{ \Tilde{\mu}_{f,\rho,q}}(a)= \lim_{n \rightarrow \infty} [p^{n}]_{\rho,q} \Tilde{\mu}_{f, \rho,q}( a+ p^n \Bbb Z_p)=\beta_{0:m}(\rho,q)\biggl(\frac{\rho}{q}\biggl)^{a} f(a),
	\end{aligned}
	\end{equation*}
	where $m$ is the degree of $f_m(x)$. Let $m$ be the sufficiently large number such that $\| f -f_m \|_{\infty} \leq p^{-n}.$ Then we have
	\begin{equation*}
	\begin{aligned}
[]	[p^{n}]_{\rho,q}\Tilde{\mu}_{f,\rho,q}( a+ p^n \Bbb Z_p) & = [p^{n}]_{\rho,q} \mu_{f_m,\rho,q}( a+ p^n \Bbb Z_p)+ [p^{n}]_{\rho,q}\Tilde{\mu}_{f-f_m,\rho,q}( a+ p^{n} \Bbb Z_p)\\
	& = [p^{n}]_{\rho,q}\mu_{f_m,\rho,q}( a+ p^n \Bbb Z_p) \\
	&= \biggl(\frac{\rho}{q}\biggl)^{a}  f_m(a)\beta_{0:m}(\rho,q) \pmod {[p^n]_{\rho,q}}.
	\end{aligned}
	\end{equation*}
	For $g \in UD( \Bbb Z_p)$, we obtain
	\begin{equation*}
	\begin{aligned}
	\int_{\mathbb{Z}_p} g(x) d\Tilde{\mu}_{f,\rho,q}(x)=\int_{\mathbb{Z}_p} \beta_{0:m}(\rho,q) f(x)
	g(x) d\mu_{\rho,q}(x).
	\end{aligned}
	\end{equation*}
	Assume that $f$ is the function from $ UD( \Bbb Z_p, \Bbb
	C_p)$ to $ Lip( \Bbb Z_p, \Bbb C_p)$. By the definition of
	$\mu_{\rho,q}$, we easily see that $\mu_{\rho,q}$ is a strongly $p$-adic $(\rho,q)$-distribution on $\Bbb Z_p$ and for $n \gg 0$
	\begin{equation*}
	\begin{aligned}
	\big|  f_{\mu_{\rho,q}}( a )-[p^{n}]_{\rho,q}\mu_{\rho,q}( a+ p^{n} \Bbb Z_p)  \big|  \leq C_3 p^{-\nu_p(\rho^{p^n}-q^{p^n})} ,
	\end{aligned}
	\end{equation*}
	where $C_3$ is some positive constant.
		If  $\mu_{1, \rho,q}$ is associated strongly  $(\rho,q)$-measure on $\Bbb Z_p$, then we have
	\begin{equation*}
	\begin{aligned}
	\big| [p^{n}]_{\rho,q} \mu_{1,\rho,q}( a+ p^{n} \Bbb Z_p)-f_{\mu_{\rho,q}}( a ) \big|  \leq C_4 p^{-\nu_p(\rho^{p^n}-q^{p^n})} ,
	\end{aligned}
	\end{equation*}
	where   $n \gg 0$ and $C_4$ is some positive constant.	
	Consequently, it follows that
	\begin{equation*}
	\begin{aligned}
	& \big| [p^{n}]_{\rho,q} \mu_{\rho,q}( a+ p^n \Bbb Z_p)-[p^{n}]_{\rho,q}\mu_{1,\rho,q}( a+ p^{n} \Bbb Z_p)  \big|  \\
	&  \leq  \big| [p^{n}]_{\rho,q} \mu_{\rho,q}( a+ p^n \Bbb Z_p)- f_{\mu_{\rho,q}}(a) \big|
	+   \big|  f_{\mu_{\rho,q}}(a) - [p^{n}]_{\rho,q} \mu_{1,\rho,q}( a+ p^n
	\Bbb Z_p) \big| \leq K,
	\end{aligned}
	\end{equation*}
	where  $K$ is some positive constant. 	
	Therefore, $\mu_{\rho,q}-\mu_{1,\rho,q}$ is a $p$-adic $(\rho,q)$-distribution on $\Bbb Z_p$. The proof of Theorem \ref{thm3} is complete.
\end{proof}

\subsection*{Acknowledgements}
The work of D. Lim was partially supported by the National Research Foundation of Korea (NRF) grant funded by the Korean government (MSIT) NRF-2021R1C1C1010902.


\begin{thebibliography}{99}

\bibitem{Ara}
{S. Araci, U. Duran, M. Acikgoz}, \textit{$(\rho,q)$-Volkenborn integration,} { J. Number Theory} \textbf{171} (2017), 18-30; Available online at \url{https://doi.org/10.1016/j.jnt.2016.07.019}.

\bibitem{Bauer}
H. Bauer, {Measure and integration theory. Translated from the German by Robert B. Burckel. De Gruyter Studies in Mathematics, 26,} Walter de Gruyter $\&$ Co. Berlin, 2001. 


\bibitem{L1}J. M. Calabuig, P. Gregori, E. A. Sanchez Perez, \textit{ Radon-Nikodym derivatives for vector measures belonging to Kothe function spaces},
J. Math. Anal. Appl. \textbf{348} (2008), 469--479.

\bibitem{Car}
L. Carlitz, {\it $q$-Bernoulli numbers and polynomials}, Duke Math. J. {\bf 15} (1948), 987--1000. 

\bibitem{Choi}
J. Choi, P.J. Anderson, H.M. Srivastava, {\it Carlitz’s $q$-Bernoulli and $q$-Euler numbers and polynomials and a class of generalized $q$-Hurwitz zeta functions}, Appl. Math. Comput. {\bf 215} (3) (2009) 1185-–1208.

\bibitem{Do}
Y. Do, D. Lim, \textit{On $(h,q)$-Daehee numbers and polynomials}, Adv. Difference Equ. (2015) {\bf 2015}:107, 9pp.; Available online at \url{https://doi.org/10.1186/s13662-015-0445-3}.

\bibitem{Duran}
U. Duran, M. Acikgoz, {\it On $(\rho,q)$-Euler numbers and polynomials associated with $(\rho,q)$-Volkenborn integrals}, Int. J. Number Theory, {\bf 14} (2018), no. 1, 241--253; Available online at \url{https://doi.org/10.1142/S179304211850015X}.

\bibitem{L2} E. de Amo, C.  Diaz Carrillo, \textit{A Radon-Nikodym derivative for positive linear functionals},
Studia Math.  \textbf{192} (2009), 1-14.

\bibitem{L3} K. George, \textit{On the Radon-Nikodym theorem},
Amer. Math. Monthly  \textbf{115} (2008), 556-558.

\bibitem{DKim}
D. Kim, H. Ozden, Y. Simsek, A. Yardimci, {\it New families of special numbers and polynomials arising from applications of $p$-adic $q$-integrals}, Adv. Difference Equ. (2017), Paper No. 207, 11 pp. ; Available online at \url{https://doi.org/10.1186/s13662-017-1273-4}.

\bibitem{Kim1}
{T. Kim,}
\textit{ $q$-Volkenborn integration,} {Russ. J. Math. Phys.} \textbf{9} (2002), no. 3, 288--299.

\bibitem{Kim2}
{T. Kim,}
\textit{ Lebesgue-Radon-Nikodým theorem with respect to $q$-Volkenborn distribution on $\mu_q$,}
{Appl. Math. Comput.} \textbf{187} (2007) no. 1, 266–-271; Available online at \url{https://doi.org/10.1016/j.amc.2006.08.123}.

\bibitem{Kob}
{N. Koblitz},  {$p$-adic Analysis and Zeta Functions,} {Springer-Verlag, New York Inc}, 1977.

\bibitem{Niko}
O. Nikodym, {\it Sur une g\'{e}n\'{e}ralisation des int\'{e}grales de M. J. Radon}, Fund. Math. {\bf 15}(1930), 131--179.

\bibitem{Ozd}
H. Ozden, Y. Simsek, I.N. Cangul, {\it Euler polynomials associated with $p$-adic $q$-Euler measure}, Gen. Math. {\bf 15}(2–3) (2007) 24–37.

\bibitem{Rud}
{ W. Rudin},  { Real and Complex Analysis,} { McGraw-Hill, New York}, 1987.

\bibitem{Roy}
{H. L. Royden},  { Real  Analysis,} { Prentice-Hall}, 1998. 

\bibitem{L4} T. Selke, \textit{Yet another proof of the Radon-Nikodym theorem}, Amer.Math. Monthly \textbf{109}
(2002), 74-76.

\bibitem{Vol-1} 
A. Volkenborn, {\it Ein $p$-adisches Integral und seine Anwendungen. I.} Manuscripta Math. {\bf 7} (1972), 341–-373. ; Available online at \url{https://doi.org/10.1007/BF01644073}. 

\bibitem{Vol-2} 
A. Volkenborn, {\it Ein $p$-adisches Integral und seine Anwendungen. I.} Manuscripta Math. {\bf 12} (1974), 17--46. ; Available online at \url{https://doi.org/10.1007/BF01166232}. 

\bibitem{Sim}
Y. Simsek, {\it Analysis of the $p$-adic $q$-Volkenborn integrals: an approach to generalized Apostol-type special numbers and polynomials and their applications}, Cogent Math. {\bf 3} (2016), Art. ID 1269393, 17 pp. ; Available online at \url{https://doi.org/10.1080/23311835.2016.1269393}. 
\end{thebibliography}
\end{document}